\documentclass[11pt,reqno]{amsart}

\usepackage[T1]{fontenc}
\usepackage{lmodern}
\usepackage{microtype}
\usepackage{amsmath,amssymb,mathtools}
\usepackage{enumitem}
\usepackage{xcolor}
\usepackage[colorlinks=true,linkcolor=blue!55!black,
  citecolor=blue!55!black,urlcolor=blue!55!black]{hyperref}
\hypersetup{
  pdftitle={A counterexample to the Tang--Zhang Schatten norm conjecture and sharp positive results},
  pdfauthor={Zijian Zeng; Houde Liu; Kuru Ratnavelu},
  pdfkeywords={Schatten norm, matrix absolute value, sharp constant, rank-one matrix, counterexample}
}

\newtheorem{theorem}{Theorem}[section]

\newtheorem{corollary}[theorem]{Corollary}
\theoremstyle{remark}
\newtheorem{remark}[theorem]{Remark}

\newcommand{\C}{\mathbb C}
\newcommand{\M}{\mathbb M}
\newcommand{\Tr}{\operatorname{Tr}}

\newcommand{\TZ}{C^{\mathrm{TZ}}}

\title[A counterexample to a Schatten norm conjecture]
{A counterexample to the Tang--Zhang Schatten norm conjecture\\
and sharp positive results}

\author{Zijian Zeng}
\address{Institute of Computer Science and Digital Innovation,
UCSI University, Kuala Lumpur 56000, Malaysia}
\email{1002266693@ucsiuniversity.edu.my}

\author{Houde Liu}
\address{Tsinghua Shenzhen International Graduate School,
Tsinghua University, Shenzhen 518055, China}
\email{liu.hd@sz.tsinghua.edu.cn}

\author{Kuru Ratnavelu}
\address{Institute of Computer Science and Digital Innovation,
UCSI University, Kuala Lumpur 56000, Malaysia}
\address{Institute of Mathematical Sciences, University of Malaya,
Kuala Lumpur 50603, Malaysia}
\email{Kurunathan@ucsiuniversity.edu.my}


\begin{document}

\begin{abstract}
For \(m\ge2\), let \(c_p(m)\) be the all-dimensional best constant in
\[
 \left\|\sum_{k=1}^m A_k\right\|_p
 \le c_p(m)\left\|\sum_{k=1}^m |A_k|\right\|_p .
\]
Tang and Zhang conjectured an explicit formula for every finite \(p>1\).
We disprove the conjecture with two explicit real \(2\times2\) rank-one
matrices at \(p=3/2\).  The comparison is certified by seven strict
rational inequalities and, in particular, places the attained ratio above
\(207/200\) while the conjectured constant lies below \(207/200\).
On the positive side, we prove the conjectured sharp bound for every family
of rank-at-most-one summands when \(2\le p<\infty\), and classify all equality
cases.  We also prove the corresponding endpoint statement for \(p=\infty\).
Finally, for arbitrary complex matrices we establish the conjectured sharp
constant in the case \(m=2,\ p=4\).
\end{abstract}

\maketitle

\section{Introduction}

For \(A\in\M_n(\C)\), write
\[
 |A|=(A^*A)^{1/2}
\]
and let \(\|A\|_p=(\Tr |A|^p)^{1/p}\) be the Schatten \(p\)-norm for
\(1\le p<\infty\); \(\|\cdot\|_\infty\) denotes the operator norm.
For fixed \(m,n\ge1\), first define
\begin{equation}\label{eq:def-cpmn}
 c_p^{\mathrm{abs}}(m,n)=
 \sup_{(A_1,\ldots,A_m)\ne0}
 \frac{\left\|\sum_{k=1}^m A_k\right\|_p}
      {\left\|\sum_{k=1}^m |A_k|\right\|_p},
 \qquad A_k\in\M_n(\C).
\end{equation}
We use the dimension-free notation
\begin{equation}\label{eq:def-cpm}
 c_p(m)=\sup_{n\ge1}c_p^{\mathrm{abs}}(m,n).
\end{equation}
The all-zero family is excluded; the denominator otherwise cannot vanish.

Tang and Zhang \cite{Tang_2026} proved
\[
 c_1(m)=1,\qquad
 c_2(m)=\sqrt{\frac{1+\sqrt m}{2}},\qquad
 c_\infty(m)=\sqrt m,
\]
and proposed a formula for the remaining exponents.  For finite \(p>1\),
let \(x_{p,m}>1\) be the unique solution of
\begin{equation}\label{eq:root-general}
 x^p-2x-(m-1)=0.
\end{equation}
Their conjectured value is
\begin{equation}\label{eq:TZ}
 \TZ_{p,m}
 =
 \frac{\sqrt{x_{p,m}(x_{p,m}+m-1)}}
      {(x_{p,m}^p+m-1)^{1/p}}.
\end{equation}
The lower bound \(c_p(m)\ge\TZ_{p,m}\) is attained by a rank-one
equiangular family.

The dimension parameter matters at fixed size.  Writing
\(d=\min\{m,n\}\), Zhang \cite{Zhang_2026} subsequently obtained
\[
 c_1^{\mathrm{abs}}(m,n)=1,\qquad
 c_2^{\mathrm{abs}}(m,n)=\sqrt{\frac{1+\sqrt d}{2}},\qquad
 c_\infty^{\mathrm{abs}}(m,n)=\sqrt d,
\]
together with
\(c_p^{\mathrm{abs}}(m,n)\le
d^{1/2-1/(2p)}\) for \(1\le p\le\infty\).
Bourin and Lee \cite{Bourin_2026} also highlighted the question for
Schatten exponents other than two.  Neither result asserts the explicit
formula \eqref{eq:TZ} for general \(p\).

Our first result shows that this lower bound is not the sharp constant in
general.

\begin{theorem}[Exact counterexample]\label{thm:counterexample}
For \(m=n=2\) and \(p=3/2\), there are real rank-one matrices \(A_1,A_2\)
such that
\[
 \frac{\|A_1+A_2\|_{3/2}}
      {\||A_1|+|A_2|\|_{3/2}}
 >
 \frac{207}{200}
 >
 \TZ_{3/2,2}.
\]
Consequently, the Tang--Zhang conjecture is false.
\end{theorem}

The failure occurs inside the rank-one class, but on the opposite side of
the Hilbertian exponent from the natural positive result.

\begin{theorem}[Sharp rank-one bound]\label{thm:rank-one}
Let \(m\ge2\), \(2\le p<\infty\), and let
\(A_1,\ldots,A_m\in\M_n(\C)\) have rank at most one and not all vanish.
Then
\begin{equation}\label{eq:rank-one-bound}
 \left\|\sum_{k=1}^m A_k\right\|_p
 \le
 \TZ_{p,m}
 \left\|\sum_{k=1}^m |A_k|\right\|_p .
\end{equation}
The constant is sharp in the dimension-free rank-one problem and is attained
whenever \(n\ge m\).  Equality holds precisely as
follows, up to common input and output unitaries, a common positive scale,
and the harmless phase changes in rank-one factorizations:
\[
 A_k=r\,u v_k^*,\qquad
 \langle v_j,v_k\rangle=
 \begin{cases}
 1,&j=k,\\
 s_{p,m},&j\ne k,
 \end{cases}
\]
where \(r>0\), \(u\) is a unit vector, and
\[
 s_{p,m}=\frac{x_{p,m}-1}{x_{p,m}+m-1}.
\]
In particular, equality requires \(n\ge m\).
\end{theorem}

For \(p=\infty\), the same rank-one argument gives the sharp constant
\(\sqrt m\), with the right vectors orthonormal.  Our third result leaves
the rank-one restriction entirely.

\begin{theorem}[The full case \(m=2,\ p=4\)]\label{thm:p4}
Let \(A,B\in\M_n(\C)\), and let \(x>1\) be the solution of
\[
 x^4-2x-1=0.
\]
Then
\begin{equation}\label{eq:p4-bound}
 \|A+B\|_4^4
 \le
 \frac{x^2(x+1)}2\,\||A|+|B|\|_4^4.
\end{equation}
The constant is sharp in the dimension-free problem, is attained for every
\(n\ge2\), and equals \((\TZ_{4,2})^4\).
\end{theorem}

The paper is organized as follows.  Section~\ref{sec:counterexample} gives
the exact counterexample and its rational certificate.
Section~\ref{sec:rank-one} proves the sharp rank-one theorem and its equality
statement.  Section~\ref{sec:p4} proves Theorem~\ref{thm:p4}.

\section{An exact \texorpdfstring{\(2\times2\)}{2-by-2} counterexample}
\label{sec:counterexample}

Set
\[
 e=\begin{pmatrix}1\\0\end{pmatrix},\qquad
 u=\begin{pmatrix}39/40\\ \sqrt{79}/40\end{pmatrix},\qquad
 v=\begin{pmatrix}5/8\\ \sqrt{39}/8\end{pmatrix}.
\]
These are real unit vectors.  Define
\begin{equation}\label{eq:counterexample-matrices}
 A_1=ee^T=
 \begin{pmatrix}1&0\\0&0\end{pmatrix},\qquad
 A_2=uv^T=
 \begin{pmatrix}
 39/64&39\sqrt{39}/320\\[1mm]
 \sqrt{79}/64&\sqrt{3081}/320
 \end{pmatrix}.
\end{equation}
Both matrices have rank one and unique nonzero singular value equal to one.
Consequently,
\[
 |A_1|=ee^T,\qquad |A_2|=vv^T.
\]

\begin{proof}[Proof of Theorem~\ref{thm:counterexample}]
Let \(U=(e,u)\) and \(V=(e,v)\).  Then \(A_1+A_2=UV^T\), and the two
Gram matrices are
\[
 L=U^TU=
 \begin{pmatrix}1&39/40\\39/40&1\end{pmatrix},
 \qquad
 G=V^TV=
 \begin{pmatrix}1&5/8\\5/8&1\end{pmatrix}.
\]
The squared singular values of \(UV^T\) are the eigenvalues of \(LG\).
The two Gram matrices are simultaneously diagonalized by
\((1,1)^T\) and \((1,-1)^T\), so those squared singular values are
\begin{equation}\label{eq:counterexample-sv}
 (1+39/40)(1+5/8)=\frac{1027}{320},
 \qquad
 (1-39/40)(1-5/8)=\frac3{320}.
\end{equation}
On the other hand, the eigenvalues of
\(|A_1|+|A_2|=VV^T\) are
\begin{equation}\label{eq:counterexample-den-eigs}
 \frac{13}{8},\qquad \frac38.
\end{equation}
Thus, if
\[
 R=\frac{\|A_1+A_2\|_{3/2}}
         {\||A_1|+|A_2|\|_{3/2}},
\]
then
\begin{equation}\label{eq:Rpower}
 R^{3/2}=
 \frac{(1027/320)^{3/4}+(3/320)^{3/4}}
      {(13/8)^{3/2}+(3/8)^{3/2}}.
\end{equation}

We first prove \(R>207/200\) using rational arithmetic only.  Positivity
allows us to raise each proposed enclosure to the fourth or second power.
The exact differences are
\begin{align}
 \left(\frac{1027}{320}\right)^3
 -\left(\frac{11989}{5000}\right)^4
 &=\frac{29144230879351}{40000000000000000}>0,\label{eq:cert1}\\
 \left(\frac3{320}\right)^3
 -\left(\frac{301}{10000}\right)^4
 &=\frac{124819571}{40000000000000000}>0,\label{eq:cert2}\\
 \left(\frac{4143}{2000}\right)^2
 -\left(\frac{13}{8}\right)^3
 &=\frac{773}{8000000}>0,\label{eq:cert3}\\
 \left(\frac{2297}{10000}\right)^2
 -\left(\frac38\right)^3
 &=\frac{5543}{200000000}>0.\label{eq:cert4}
\end{align}
Put
\[
 A_0=\frac{11989}{5000}+\frac{301}{10000}
 =\frac{24279}{10000},\qquad
 B_0=\frac{4143}{2000}+\frac{2297}{10000}
 =\frac{5753}{2500},
\]
and \(q=207/200\).  A final exact comparison gives
\begin{equation}\label{eq:cert5}
 A_0^2-q^3B_0^2
 =\frac{1172956601313}{50000000000000}>0.
\end{equation}
Equations~\eqref{eq:cert1}--\eqref{eq:cert5} imply
\[
 R^{3/2}>\frac{A_0}{B_0}>q^{3/2},
\]
and hence \(R>q\).

It remains to put the conjectured constant below the same rational separator.
Let \(x>1\) solve \(x^{3/2}-2x-1=0\), and write \(x=t^2\).
The unique positive root \(t\) of
\begin{equation}\label{eq:t-cubic}
 h(t)=t^3-2t^2-1=0
\end{equation}
lies above \(4/3\), where \(h\) is strictly increasing.  Moreover,
\begin{equation}\label{eq:cert6}
 h\left(\frac{1103}{500}\right)
 =\frac{310727}{125000000}>0,
\end{equation}
so \(t<T:=1103/500\).

Writing \(C=\TZ_{3/2,2}\) and using \(t^3=2t^2+1\), we obtain
\[
 C^6=\frac{t^6}{16(t^2+1)}=:H(t).
\]
The function \(H\) is strictly increasing for \(t>0\), since
\[
 H'(t)=\frac{t^5(2t^2+3)}{8(t^2+1)^2}>0.
\]
The remaining exact difference is
\begin{equation}\label{eq:cert7}
 q^6-H(T)
 =
 \frac{26731151399029597}
 {18772595200000000000}>0.
\end{equation}
It follows that \(C^6<H(T)<q^6\), hence \(C<q<R\).
\end{proof}

\begin{remark}
Numerically,
\[
 R=1.0364136587048904\ldots,\qquad
 \TZ_{3/2,2}=1.0346539518514341\ldots.
\]
These decimals play no role in the proof.
\end{remark}

\section{The sharp rank-one problem for
\texorpdfstring{\(p\ge2\)}{p greater than or equal to 2}}
\label{sec:rank-one}

We now prove Theorem~\ref{thm:rank-one}.  The proof reduces the matrix
problem to a single scalar variable.

\begin{proof}[Proof of Theorem~\ref{thm:rank-one}]
Write
\[
 A_k=r_k u_kv_k^*,\qquad r_k\ge0,
\]
where \(u_k,v_k\) are unit vectors whenever \(r_k>0\).  Form the column
matrices
\[
 U=(\sqrt{r_1}u_1,\ldots,\sqrt{r_m}u_m),\qquad
 V=(\sqrt{r_1}v_1,\ldots,\sqrt{r_m}v_m)
\]
and their Gram matrices
\[
 L=U^*U,\qquad G=V^*V.
\]
Then
\begin{equation}\label{eq:rank-one-factorization}
 \sum_k A_k=UV^*,\qquad
 \sum_k|A_k|=VV^*,
\end{equation}
and
\begin{equation}\label{eq:common-diagonal}
 \operatorname{diag}L=\operatorname{diag}G=(r_1,\ldots,r_m),\qquad
 \Tr L=\Tr G=:R>0.
\end{equation}

Let
\[
 K=L^{1/2}GL^{1/2}.
\]
The nonzero eigenvalues of \(K\) are the squared singular values of \(UV^*\);
the nonzero eigenvalues of \(G\) are the eigenvalues of \(VV^*\).
Therefore
\begin{equation}\label{eq:rank-one-traces}
 \left\|\sum_k A_k\right\|_p^p=\Tr K^{p/2},\qquad
 \left\|\sum_k|A_k|\right\|_p^p=\Tr G^p.
\end{equation}
Since \(p/2\ge1\), the nonnegative eigenvalues of \(K\) give
\begin{equation}\label{eq:trace-power}
 \Tr K^{p/2}\le(\Tr K)^{p/2}.
\end{equation}
If \(\lambda=\lambda_{\max}(G)\), then
\begin{equation}\label{eq:trace-LG}
 \Tr K=\Tr(LG)\le \lambda\Tr L=R\lambda.
\end{equation}

List the \(m\) eigenvalues of \(G\), including zeros, as
\(\lambda=\lambda_1\ge\lambda_2\ge\cdots\ge\lambda_m\ge0\).
Convexity gives
\begin{equation}\label{eq:denominator-convexity}
 \Tr G^p
 \ge
 \lambda^p+\frac{(R-\lambda)^p}{(m-1)^{p-1}}.
\end{equation}
If \(\lambda=R\), then \eqref{eq:rank-one-traces}--\eqref{eq:trace-LG}
give a ratio at most one, which is strictly smaller than the desired sharp
constant.  Suppose henceforth that \(\lambda<R\), and put
\[
 y=\frac{(m-1)\lambda}{R-\lambda}.
\]
Since \(\lambda\ge R/m\), one has \(y\ge1\).  Combining
\eqref{eq:rank-one-traces}--\eqref{eq:denominator-convexity} yields
\begin{equation}\label{eq:scalar-reduction}
 \frac{\|\sum_kA_k\|_p}{\|\sum_k|A_k|\|_p}
 \le
 F_{p,m}(y)
 :=
 \frac{\sqrt{y(y+m-1)}}{(y^p+m-1)^{1/p}}.
\end{equation}
Direct differentiation gives
\[
 \frac{d}{dy}\log F_{p,m}(y)
 =
 \frac{(m-1)(2y+m-1-y^p)}
 {2y(y+m-1)(y^p+m-1)}.
\]
For \(p\ge2\), the function \(y^p-2y-(m-1)\) is strictly increasing on
\([1,\infty)\), apart from an inessential zero derivative at the left
endpoint when \(p=2\).  It has exactly one zero \(x_{p,m}>1\).
Thus \(F_{p,m}\) has a unique maximum at \(x_{p,m}\), and
\eqref{eq:rank-one-bound} follows from \eqref{eq:TZ}.

We next track equality.  Put
\[
 \beta=\frac{R-\lambda}{m-1}.
\]
Equality in the scalar maximization and in
\eqref{eq:denominator-convexity} forces
\[
 \lambda=x_{p,m}\beta,\qquad
 \operatorname{spec}(G)=\{\lambda,\beta,\ldots,\beta\}.
\]
The top eigenspace is one-dimensional.  Equality in \eqref{eq:trace-LG}
forces the range of \(L\) into this eigenspace.  If \(w\) is its unit
eigenvector, then
\begin{equation}\label{eq:equality-gram}
 L=Rww^*,\qquad
 G=\beta I+(\lambda-\beta)ww^*.
\end{equation}
The common diagonal condition \eqref{eq:common-diagonal} now gives
\[
 R|w_k|^2
 =
 \beta+(\lambda-\beta)|w_k|^2.
\]
Because \(R=\lambda+(m-1)\beta\), it follows that
\[
 |w_k|^2=\frac1m,\qquad r_k=\frac Rm
 \quad(1\le k\le m).
\]
After simultaneous phase changes in the factorizations of the \(A_k\),
we may take \(w=m^{-1/2}(1,\ldots,1)^T\).  Equation
\eqref{eq:equality-gram} then says that the \(u_k\) coincide and
\[
 \langle v_j,v_k\rangle
 =
 \frac{\lambda-\beta}{R}
 =
 \frac{x_{p,m}-1}{x_{p,m}+m-1}
 \quad(j\ne k).
\]
Conversely, this family takes equality at every step.  The displayed Gram
matrix is positive definite, so equality requires \(n\ge m\).
\end{proof}

\begin{corollary}[Rank-one endpoint]\label{cor:rank-one-infinity}
If \(A_1,\ldots,A_m\) have rank at most one, then
\[
 \left\|\sum_k A_k\right\|_\infty
 \le\sqrt m\left\|\sum_k|A_k|\right\|_\infty.
\]
The constant is sharp in the dimension-free sense.  Equality is possible only
when \(n\ge m\), and then holds precisely for equal nonzero singular values, a
common one-dimensional range, and pairwise orthogonal right vectors, modulo
the same unitary and phase symmetries as above.
\end{corollary}

\begin{proof}
Use the notation in the preceding proof and set
\(\lambda=\lambda_{\max}(G)\).  Then
\[
 \|UV^*\|_\infty
 \le\|U\|_\infty\|V^*\|_\infty
 =\sqrt{\|L\|_\infty\lambda}
 \le\sqrt{R\lambda}.
\]
Since \(R=\Tr G\le m\lambda\), while
\(\|VV^*\|_\infty=\lambda\), the asserted inequality follows.
Equality requires \(L\) to have rank one and \(G=\lambda I_m\).
The common diagonal condition then gives exactly the stated configuration.
\end{proof}

\section{The full \texorpdfstring{\(m=2,\ p=4\)}{m=2, p=4} problem}
\label{sec:p4}

\begin{proof}[Proof of Theorem~\ref{thm:p4}]
Put \(H=|A|\) and \(K=|B|\).  Extend the partial isometries in the polar
decompositions to unitaries, and write \(A=UH,\ B=VK\).
With \(W=U^*V\), unitary invariance reduces the numerator to
\[
 S=H+WK.
\]
Set
\[
 a=\Tr H^4,\quad b=\Tr K^4,\quad d=\Tr H^2K^2.
\]
The Hilbert--Schmidt triangle inequality applied to
\[
 SS^*=H^2+HKW^*+WKH+WK^2W^*
\]
gives
\begin{equation}\label{eq:p4-numerator}
 \|S\|_4^4
 \le(\sqrt a+\sqrt b+2\sqrt d)^2.
\end{equation}

Introduce
\[
 e=\Tr H^3K,\qquad f=\Tr HK^3,\qquad g=\Tr HKHK.
\]
Cyclically collecting all words in the noncommutative expansion gives
\begin{equation}\label{eq:p4-expansion}
 \Tr(H+K)^4=a+b+4(e+f)+4d+2g.
\end{equation}
Two lower bounds are needed.  First,
\begin{equation}\label{eq:p4-positive}
 e+f-2d
 =
 \Tr\big((H-K)H(H-K)K\big)
 =
 \|H^{1/2}(H-K)K^{1/2}\|_2^2
 \ge0.
\end{equation}
Second, the three-factor Schatten H\"older inequality gives
\begin{align*}
 \sqrt d=\|HK\|_2
 &=
 \|H^{1/2}(H^{1/2}K^{1/2})K^{1/2}\|_2\\
 &\le
 \|H^{1/2}\|_8
 \|H^{1/2}K^{1/2}\|_4
 \|K^{1/2}\|_8
 =
 a^{1/8}g^{1/4}b^{1/8}.
\end{align*}
Thus, when \(ab>0\),
\begin{equation}\label{eq:g-lower}
 g\ge\frac{d^2}{\sqrt{ab}}.
\end{equation}
If \(ab=0\), one of \(H,K\) vanishes and the theorem is immediate.
Combining \eqref{eq:p4-expansion}--\eqref{eq:g-lower}, we obtain
\begin{equation}\label{eq:p4-denominator}
 \Tr(H+K)^4
 \ge a+b+12d+\frac{2d^2}{\sqrt{ab}}.
\end{equation}

Let \(A_0=\sqrt a,\ B_0=\sqrt b\), and define
\[
 z=\frac{A_0+B_0}{\sqrt{A_0B_0}}\ge2,\qquad
 s=\sqrt{\frac{d}{A_0B_0}}\in[0,1].
\]
The upper bound on \(s\) is the Hilbert--Schmidt Cauchy--Schwarz
inequality \(d\le\sqrt{ab}\).  From
\eqref{eq:p4-numerator} and \eqref{eq:p4-denominator},
\begin{equation}\label{eq:Phi}
 \frac{\|A+B\|_4^4}{\||A|+|B|\|_4^4}
 \le
 \Phi(z,s)
 :=
 \frac{(z+2s)^2}{z^2-2+12s^2+2s^4}.
\end{equation}

If \(s\ge1/2\), then
\[
 2(z^2-2+12s^2+2s^4)-(z+2s)^2
 =
 (z-2s)^2+16s^2+4s^4-4\ge0,
\]
so \(\Phi(z,s)\le2\).

Suppose \(0\le s\le1/2\).  The sign of
\(\partial\Phi/\partial z\) is the sign of
\[
 s^4+6s^2-zs-1.
\]
Since \(z\ge2\),
\[
 s^4+6s^2-zs-1
 \le s^4+6s^2-2s-1
 \le2s-1\le0;
\]
the middle inequality follows from
\(s^3+6s-4\le1/8+3-4<0\).
Consequently,
\begin{equation}\label{eq:f-s}
 \Phi(z,s)\le\Phi(2,s)
 =
 f(s):=\frac{2(1+s)^2}{1+6s^2+s^4}.
\end{equation}
The sign of \(f'(s)\) is the sign of
\[
 1-6s-2s^3-s^4.
\]
Hence \(f\) has a unique maximizer \(s_0\in(0,1/2)\), characterized by
\begin{equation}\label{eq:s0}
 s_0^4+2s_0^3+6s_0-1=0.
\end{equation}
The first branch cannot dominate, because \(f(1/6)>2\).

Set \(x=(1+s_0)/(1-s_0)\).  A direct calculation gives
\[
 (x^4-2x-1)(1-s_0)^4
 =
 2(s_0^4+2s_0^3+6s_0-1)=0,
\]
and
\begin{equation}\label{eq:f-constant}
 f(s_0)
 =
 \frac{x^2(x+1)^2}{x^4+1}
 =
 \frac{x^2(x+1)}2,
\end{equation}
where the last equality uses \(x^4+1=2(x+1)\).
Equations~\eqref{eq:Phi}--\eqref{eq:f-constant} prove
\eqref{eq:p4-bound}.

For sharpness, choose unit vectors \(r_1,r_2\) with
\(\langle r_1,r_2\rangle=s_0\), choose a unit vector \(u\), and set
\[
 A=ur_1^*,\qquad B=ur_2^*.
\]
Then
\[
 \|A+B\|_4^4=4(1+s_0)^2
\]
and
\[
 \||A|+|B|\|_4^4
 =(1+s_0)^4+(1-s_0)^4
 =2(1+6s_0^2+s_0^4).
\]
Their ratio is \(f(s_0)\), proving sharpness.
\end{proof}

\begin{remark}
Orthogonal direct sums of the two-dimensional extremal block give
higher-dimensional equality examples.  The proof above does not attempt a
complete classification of all equality cases for Theorem~\ref{thm:p4}.
\end{remark}

\section*{Reproducibility and disclosure}

The exact counterexample certificate consists of the seven positive rational
differences in \eqref{eq:cert1}--\eqref{eq:cert7}; it requires no numerical
linear algebra.  A standard-library verification script accompanies this
manuscript.  Low-dimensional numerical searches were used for exploration and
adversarial testing only.

OpenAI Codex assisted with proof exploration, counterexample search,
adversarial checking, and manuscript preparation.  The submitting author is
responsible for the correctness of every statement and for compliance with
the target journal's authorship and disclosure policies.

\bibliographystyle{plain}
\bibliography{references}

\end{document}